\documentclass{amsart}
\usepackage{amsthm,stmaryrd,amsmath,amscd,enumerate}
\usepackage{amssymb}
\usepackage{epsfig}
\usepackage[usenames,dvipsnames]{color}
\usepackage{verbatim}
\usepackage{ subfig}
\usepackage{dsfont}
\usepackage{mathrsfs,float}
\usepackage[all]{xy}
\usepackage[active]{srcltx}
\usepackage[utf8]{inputenc}
\usepackage{fullpage}

\usepackage{url}
\usepackage[bookmarks=true,%
    colorlinks=true,%
    linkcolor=blue,%
    citecolor=blue,%
    filecolor=blue,%
    menucolor=blue,%
    urlcolor=blue,%
    breaklinks=true,
		bookmarksdepth=2]{hyperref}

\renewcommand{\dag}[1]{{#1}^\dagger}
\newcommand{\dagp}[1]{\bp{#1}^\dagger}

\newcommand{\Cob}{{\mathcal{C}ob}}

\newcommand{\A}{{\mathsf A}}
\newcommand{\Proj}{{\mathsf {Proj}}}

\newcommand{\Mat}{\operatorname{Mat}}

\newcommand{\lk}{\operatorname{lk}}

\newcommand{\coev}{\stackrel{\longrightarrow}{\operatorname{coev}}}
\newcommand{\ev}{\stackrel{\longrightarrow}{\operatorname{ev}}}
\newcommand{\tev}{\stackrel{\longleftarrow}{\operatorname{ev}}}
\newcommand{\tcoev}{\stackrel{\longleftarrow}{\operatorname{coev}}}
\newcommand{\brk}[1]{{\left\langle{#1}\right\rangle}}

\newcommand{\FK}{\ensuremath{\mathbb{K}} }
\newcommand{\ve}{\varepsilon}

\newcommand{\ro}{r}
\newcommand{\g}{\ensuremath{\mathfrak{g}}}

\newcommand{\La}{{\mathscr L}}

\newcommand{\X}{{\mathsf X}}

\newcommand{\Zr}{{\mathsf Z}}

\newcommand{\qr}{{q}}
\newcommand{\coh}{\omega}
\newcommand{\vp}{\varphi}

\newcommand{\V}{\mathsf{V}}
\newcommand{\VV}{\mathbb{V}}

\newcommand{\hS}{\wh{S}}
\newcommand{\Su}{\Sigma}
\newcommand{\dSu}{\wt\Sigma}
\newcommand{\dM}{\wt M}

\newcommand{\e}{{\operatorname{e}}}
\newcommand{\slt}{{\mathfrak{sl}(2)}}
\newcommand{\Uq}{{U_q\slt}}
\newcommand{\UqMed}{{\wb U_q\slt}}

\newcommand{\UsltH}{{U_q^{H}\slt}}
\newcommand{\Ubar}{{\wb U_q^{H}\slt}}
\newcommand{\unit}{\ensuremath{\mathbb{I}}}
\newcommand{\cat}{\mathscr{C}}
\newcommand{\catd}{\mathscr{D}}

\newcommand{\Id}{\operatorname{Id}}
\newcommand{\bp}[1]{{\left(#1\right)}}

\newcommand{\qn}[1]{{\left\{#1\right\}}}
\newcommand{\qN}[1]{{\left[#1\right]}}

\newcommand{\qd}{{\mathsf d}}
\newcommand{\qdim}{\operatorname{qdim}}
\newcommand{\End}{\operatorname{End}}
\newcommand{\Aut}{\operatorname{Aut}}
\newcommand{\Hom}{\operatorname{Hom}}

\newcommand{\Span}{\operatorname{Span}}

\newcommand{\ptr}{\operatorname{ptr}}
\newcommand{\C}{\ensuremath{\mathbb{C}} }
\newcommand{\Z}{\ensuremath{\mathbb{Z}} }
\newcommand{\R}{\ensuremath{\mathbb{R}} }
\newcommand{\N}{\ensuremath{\mathbb{N}} }
\newcommand{\Q}{\ensuremath{\mathbb{Q}} }
\newcommand{\wt}{\widetilde}
\newcommand{\wh}{\widehat}
\newcommand{\wb}{\overline}

\newcommand{\Cp}{{\ddot\C}}

\newcommand{\dep}{\delta}
\newcommand{\et}{{\quad\text{and}\quad}}

\renewcommand{\vec}{w}
\newcommand{\vecl}{\vec^L}
\newcommand{\vecr}{\vec^R}
\newcommand{\vech}{\vec^H}
\newcommand{\vecs}{\vec^S}

\newcommand{\md}{\operatorname{\mathsf{d}}}
\newcommand{\mt}{\operatorname{\mathsf{t}}}

\newcommand{\ob}{\operatorname{Ob}(\cat)}
\newcommand{\mathsmall}[1]{\mbox{\small$#1$}}
\newcommand{\sqr}{\operatorname{sqr}}
\newcommand{\catHe}{\cat^{\dagger}}
\newcommand{\catdHe}{\catd^{\dagger}}
\newcommand{\con}[1]{\bar{#1}^{\dagger}}
\def\cal#1{\mathcal{#1}}%

\newtheorem{theo}{Theorem}[section]
\newtheorem*{theo*}{Theorem}
\newtheorem{lemma}[theo]{Lemma}
\newtheorem{prop}[theo]{Proposition}
\newtheorem{cor}[theo]{Corollary}

\theoremstyle{definition}
\newtheorem{defi}[theo]{Definition}
\newtheorem{rem}[theo]{Remark}
\newtheorem{example}[theo]{Example}
\newtheorem{exo}[theo]{Exercise}

\theoremstyle{remark}

\newcounter{exo} \newcounter{numexercice}
\renewcommand{\theexo}{\arabic{exo}}

 \newcommand{\maps}{\colon}

\usepackage[all]{xy}
\usepackage{tikz}
\usetikzlibrary{decorations.markings}
\usetikzlibrary{decorations.pathreplacing}
\tikzstyle directed=[postaction={decorate,decoration={markings,
    mark=at position #1 with {\arrow{>}}}}]

\tikzset{middlearrow/.style={
        decoration={markings,
            mark= at position 0.5 with {\arrow{#1}} ,
        },
        postaction={decorate}
    }
}

\begin{document}
\title{A Hermitian TQFT from a non-semisimple category of quantum $\slt$-modules}

\author[N. Geer]{Nathan Geer}
\address{Mathematics \& Statistics\\
  Utah State University \\
  Logan, Utah 84322, USA}
  \email{nathan.geer@gmail.com}
\author[A.D. Lauda]{Aaron D. Lauda}
\address{Department of Mathematics\\
 University of Southern California \\
  Los Angeles, California 90089, USA}
  \email{lauda@usc.edu}

%
\author[B. Patureau-Mirand]{Bertrand Patureau-Mirand}
\address{UMR 6205, LMBA,\\
  universit\'e de Bretagne-Sud, BP 573,\\
  56017 Vannes, France }
\email{bertrand.patureau@univ-ubs.fr}
\author[J. Sussan]{Joshua Sussan}
\address{Department of Mathematics\\
  CUNY Medgar Evers \\
  Brooklyn, NY 11225, USA}
  \email{jsussan@mec.cuny.edu}
  \address{Mathematics Program\\
 The Graduate Center, CUNY \\
  New York, NY 10016, USA}
  \email{jsussan@gc.cuny.edu}
\begin{abstract}
We endow a non-semisimple category of modules of unrolled quantum $\slt$ with a Hermitian structure.  We also prove that the TQFT constructed in arXiv:1202.3553 using this category is Hermitian.  This gives rise to projective representations of the mapping class group in the group of indefinite unitary matrices.
  \end{abstract}

\maketitle
\setcounter{tocdepth}{3}

\section{Introduction}


Unitary topological quantum field theories  are closely related to various physical systems.  In particular, they are connected to topological phases of matter~\cite{KKR,LW,MR2608953, kirillov2011stringnet}.
There is hope that these mathematical theories could be realized physically and perhaps be used for fault tolerant quantum computing~\cite{Kit,FKLW,NKK}.   Many fundamental examples of these unitary TQFTs come from the representation theory of quantum groups at roots of unity.  The standard procedure is to perform some semisimplification on a category of representations and then use the resulting category to construct $3$-manifold invariants and their extensions to $(2+1)$-TQFTs.

If one instead works with the full representation category of a quantum group at a root of unity, without passing to the semisimplificiation, it was not obvious how to construct TQFTs, as the standard quantum trace on projective modules for quantum groups at roots of unity vanish.  The first and third authors, in collaboration with Kujawa, introduced in \cite{GKP1} a modified trace. This modified trace has the remarkable property that it does not vanish on projective objects and retains most of the important properties of the standard quantum trace.  This construction led to new link invariants, $3$-manifold invariants  \cite{CGP1}, and $(2+1)$-TQFTs \cite{BCGP1}.  We refer to these invariants collectively as \emph{non-semisimple} invariants.

Unlike the usual semisimple theory where quantum dimensions of simple objects are strictly positive, the modified dimensions of many objects in the non-semisimple theory are real, but not positive.  This means that there is no hope that the TQFT constructed in \cite{BCGP1} is unitary.  However, in this note we show that the TQFT arising from a non-semisimple category $\catdHe$ of representations of the unrolled quantum group for $\mathfrak{sl}_2$  is Hermitian.  This means that the TQFT will produce nondegenerate bilinear forms with an indefinite signature.

The notion of a Hermitian ribbon category was introduced by Turaev \cite{Tu}.  One of our main results is the following.
 \begin{theo*}$\catdHe$ is a \emph{Hermitian}
   ribbon category in the sense of Definition~\ref{def:Hermition}.
 \end{theo*}

We stress that analogous fundamental results for the semisimplified categories coming from quantum groups were achieved by Kirillov \cite{Kir96} and Wenzl \cite{Wen98}.  In order to accommodate projective objects in $\catdHe$, we needed to modify some arguments of \cite{Kir96, Wen98}.

Finally, in the last section we apply the Hermicity of $\catdHe $ to show that the TQFT constructed in \cite{BCGP1} is Hermitian in the sense of Turaev \cite{Tu}.

\begin{theo*}
The TQFT $(\VV,\mathsf{Z}) $ introduced in \cite{BCGP1} is Hermitian.
\end{theo*}

In Proposition~\ref{prop:mapping} we show that this implies that the mapping class group action induced by the non-semisimple TQFT produces a projective representation in the group of indefinite unitary matrices.

Even with an indefinite normed inner product, the  Hermitian  
TQFTs defined here may have physical relevance.  Indeed, quantum mechanics with indefinite norms have been studied going back to Dirac~\cite{Dirac} and Pauli~\cite{Pauli}.  Even with the indefinite norms, they observed a formalism consistent with deterministic quantum mechanics, including positive energy eigenvalues, normalizable wave functions, and time evolution by an exponential of the Hamiltonian that is self-adjoint in the indefinite norm.  More recently, the study of pseudo-Hermitian quantum mechanics has been intensively studied~\cite{Mostafazadeh_2002,Mostafazadeh_2020}, motivated by connections to $\cal{PT}$-symmetric quantum theory~\cite{Bender_2005}.  In all of these studies, indefinite normed Hilbert spaces admit Hamiltonians that are Hermitian with respect to the indefinite inner product, yet still have real spectrum, and unitary evolution.   In forthcoming work, we will show that the TQFTs studied here give rise to topological phases fitting into this framework.

\subsection{Acknowledgements}
N.G. is supported by NSF DMS-1664387.
A.D.L. is partially supported by NSF grant DMS-1902092 and Army Research Office W911NF-20-1-0075.
J.S. is partially supported by the NSF grant DMS-1807161 and PSC CUNY Award 64012-00 52.

\section{A quantization of $\slt$ and its associated ribbon
  category}\label{S:QUantSL2H}
In this section we recall the algebra $\Ubar$ and a category of
modules over this algebra.  Fix a positive integer $r$.  Let $r'=r$ if $r$ is odd and $r'=\frac{r}{2}$ otherwise.  Let $\C$ be
the complex numbers and $\Cp=(\C\setminus \Z)\cup r\Z.$ Let
$q=e^\frac{\pi\sqrt{-1}}{r}$ be a $2r^{th}$-root of unity.  We use the
notation $q^x=e^{\frac{\pi\sqrt{-1} x}{r}}$.  For $n\in \N$, we also set
 $$\qn{x}=q^x-q^{-x},\quad\qN{x}=\frac{\qn x}{\qn1},\quad\qn{n}!=\qn{n}\qn{n-1}\cdots\qn{1}\et\qN{n}!=\qN{n}\qN{n-1}\cdots\qN{1}.$$

\subsection{The Drinfel'd-Jimbo quantum group}
Let $\Uq$ be the $\C$-algebra given by generators $E, F, K, K^{-1}$
and relations:
\begin{align}\label{E:RelDCUqsl}
  KK^{-1}&=K^{-1}K=1, & KEK^{-1}&=q^2E, & KFK^{-1}&=q^{-2}F, &
  [E,F]&=\frac{K-K^{-1}}{q-q^{-1}}.
\end{align}
The algebra $\Uq$ is a Hopf algebra where the coproduct, counit and
antipode are defined by
\begin{align}\label{E:HopfAlgDCUqsl}
  \Delta(E)&= 1\otimes E + E\otimes K,
  &\varepsilon(E)&= 0,
  &S(E)&=-EK^{-1},
  \\
  \Delta(F)&=K^{-1} \otimes F + F\otimes 1,
  &\varepsilon(F)&=0,& S(F)&=-KF,
    \\
  \Delta(K)&=K\otimes K
  &\varepsilon(K)&=1,
  & S(K)&=K^{-1}
.\label{E:HopfAlgDCUqsle}
\end{align}
Let $\UqMed$ be the algebra $\Uq$ modulo the relations
$E^\ro=F^\ro=0$.

 \subsection{A modified version of $\Uq$}\label{SS:UqH}  Let $\UsltH$ be the
$\C$-algebra given by generators $E, F, K, K^{-1}, H$ and
relations in \eqref{E:RelDCUqsl} along with the relations:
\begin{align*}
  HK&=KH,
& [H,E]&=2E, & [H,F]&=-2F.
\end{align*}
The algebra $\UsltH$ is a Hopf algebra where the coproduct, counit and
antipode are defined in
\eqref{E:HopfAlgDCUqsl}--\eqref{E:HopfAlgDCUqsle} and by
\begin{align*}
  \Delta(H)&=H\otimes 1 + 1 \otimes H,
  & \varepsilon(H)&=0,
  &S(H)&=-H.
\end{align*}
Define $\Ubar$ to be the Hopf algebra $\UsltH$ modulo the relations
$E^\ro=F^\ro=0$.

Let $V$ be a finite dimensional $\Ubar$-module.  An eigenvalue
$\lambda\in \C$ of the operator $H:V\to V$ is called a \emph{weight}
of $V$ and the associated eigenspace is called a \emph{weight space}.
A vector $v$ in the $\lambda$-eigenspace
of $H$ is a \emph{weight vector} of \emph{weight} $\lambda$, i.e. $Hv=\lambda
v$.  We call $V$ a \emph{weight module} if $V$ splits as a direct sum
of weight spaces and $\qr^H=K$ as operators on $V$, i.e. $Kv=q^\lambda
v$ for any vector $v$ of weight $\lambda$.  Let $\cat$ be the category
of finite dimensional weight $\Ubar$-modules.

Since $\Ubar$ is a Hopf algebra, $\cat$ is a tensor category where
the unit $\unit$ is the 1-dimensional trivial module $\C$.  Moreover,
$\cat$ is $\C$-linear: hom-sets are $\C$-modules, the composition and
tensor product of morphisms are $\C$-bilinear, and
$\End_\cat(\unit)=\C\Id_\unit$.  When it is clear, we denote the unit
$\unit$ by $\C$.  We say a module $V$ is \emph{simple} if it has no
proper submodules.
For a module $V$ and a morphism $f\in\End_\cat(V)$, we write
$\brk f_V=\lambda\in\C$ if $f-\lambda\Id_V$ is nilpotent.  If $V$ is
simple, then Schur's lemma implies that $\End_\cat(V)=\C\Id_V$. Thus for
$f\in \End_\cat(V)$, we have $f=\brk{f}_V \Id_V$.


We will now recall the fact that the category $\cat$ is a ribbon category.
  Let $V$ and $W$ be
objects of $\cat$.  Let $\{v_i\}$ be a basis of $V$ and $\{v_i^*\}$ be
a dual basis of $V^*=\Hom_\C(V,\C)$.  Then
\begin{align*}
  \coev_V :& \C \rightarrow V\otimes V^{*}, \text{ given by } 1 \mapsto \sum
  v_i\otimes v_i^*,  &
  \ev_V: & V^*\otimes V\rightarrow \C, \text{ given by }
  f\otimes w \mapsto f(w)
\end{align*}
are duality morphisms of $\cat$.
In \cite{Oh}, Ohtsuki truncates the usual formula of the $h$-adic
quantum $\slt$ $R$-matrix to define an operator on $V\otimes W$ by
\begin{equation}
  \label{eq:R}
  R=\qr^{H\otimes H/2} \sum_{n=0}^{\ro-1} \frac{\{1\}^{2n}}{\{n\}!}\qr^{n(n-1)/2}
  E^n\otimes F^n.
\end{equation}
where $q^{H\otimes H/2}$ is the operator given by
$$q^{H\otimes H/2}(v\otimes v') =q^{\lambda \lambda'/2}v\otimes v'$$
for weight vectors $v$ and $v'$ of weights of $\lambda$ and
$\lambda'$. The $R$-matrix is not an element in $\Ubar\otimes \Ubar$.
However the action of $R$ on the tensor product of two objects of
$\cat$ is a well defined linear map.
Moreover, $R$ gives rise to a braiding $c_{V,W}:V\otimes W
\rightarrow W \otimes V$ on $\cat$ defined by $v\otimes w \mapsto
\tau(R(v\otimes w))$ where $\tau$ is the permutation $x\otimes
y\mapsto y\otimes x$.
This braiding follows from the invertibility of the $R$-matrix.  An explicit inverse
(see \cite[Section 2.1.2]{BDGG} and \cite{Oh}) is given by
\begin{equation}
  \label{eq:Rinverse}
  R^{-1}= (\sum_{n=0}^{\ro-1} (-1)^n  \frac{\{1\}^{2n}}{\{n\}!}\qr^{-n(n-1)/2}
  E^n\otimes F^n) \qr^{-H\otimes H/2}.
\end{equation}

Let $\theta$ be the operator given by
\begin{equation}
\theta=K^{\ro-1}\sum_{n=0}^{\ro-1}
\frac{\{1\}^{2n}}{\{n\}!}\qr^{n(n-1)/2} S(F^n)\qr^{-H^2/2}E^n
\end{equation}
where $q^{-H^2/2}$ is an operator defined on a weight vector $v_\lambda$ by
$q^{-H^2/2}.v_\lambda = q^{-\lambda^2/2}v_\lambda.$
Ohtsuki shows that the family of maps $\theta_V:V\rightarrow V$ in
$\cat$ defined by $v\mapsto \theta^{-1}v$ is a twist (see
\cite{jM,Oh}).

 Now the ribbon structure on $\cat$ yields right duality morphisms
\begin{equation}\label{E:d'b'}
  \tev_{V}=\ev_{V}c_{V,V^*}(\theta_V\otimes\Id_{V^*})\text{ and }\tcoev_V =(\Id_{V^*}\otimes\theta_V)c_{V,V^*}\coev_V
\end{equation}
which are compatible with the left duality morphisms $\{\coev_V\}_V$ and
$\{\ev_V\}_V$.  These duality morphisms are given explicitly by \begin{align*}
  \tcoev_{V} \maps & \C \rightarrow V^*\otimes V, \text{ where } 1 \mapsto
  \sum v_i^* \otimes K^{r-1}v_i, \\ \tev_{V} \maps & V\otimes V^*\rightarrow
  \C, \text{ where } v\otimes f \mapsto f(K^{1-r}v).
\end{align*}
The \emph{quantum dimension} $\qdim(V)$ of an object $V$ in $\cat$ is defined by
\[
\qdim(V)= \brk{\tev_V\circ \coev_V}_\unit=\sum  v_i^*(K^{1-r}v_i) \ .
\]

 For $g\in\C/2\Z$, define
$\cat_{g}$ as the full subcategory of weight modules whose weights
are all in the class  $g$ (mod $2\Z$).
Then $\cat=\{\cat_g\}_{g\in \C/2\Z}$ is a $\C/2\Z$-graded category (where
$\C/2\Z$ is an additive group). Let $V\in\cat_g$ and $V'\in\cat_{g'}$.
Then the weights of $V\otimes V'$ are congruent to $g+g' \mod 2\Z$,
and so the tensor product is in $\cat_{g+g'}$.  Also, if $g\neq g'$
 then $\Hom_\cat(V, V')=0$ since morphisms in $\cat$  preserve weights.
Finally, if $f\in V^*=\Hom_\C(V,\C)$, then by definition the action of
$H$ on $f$ is given by $(Hf)(v)=f(S(H)v)=-f(Hv)$
and so
$V^{*}\in\cat_{-g}$.

We now consider the following class of finite dimensional highest weight modules.
For each $\alpha\in \C$, we let $V_\alpha$ be the $r$-dimensional
highest weight $\Ubar$-module of highest weight $\alpha + r-1$.  The
module $V_\alpha$ has a basis $\{v_0,\ldots,v_{r-1}\}$ whose action is
given by
\begin{equation}\label{E:BasisV}
H.v_i=(\alpha + r-1-2i) v_i,\quad E.v_i= \frac{\qn i\qn{i-\alpha}}{\qn1^2}
v_{i-1} ,\quad F.v_i=v_{i+1}.
\end{equation}
For all $\alpha\in \C$, the quantum dimension of $V_\alpha$ is zero:
$$\qdim(V_\alpha)= \sum_{i=0}^{r-1} v_i^*(K^{1-r}v_i)=
 \sum_{i=0}^{r-1} q^{(r-1)(\alpha + r-1-2i)} =
 q^{(r-1)(\alpha + r-1)}\frac{1-q^{2r}}{1-q^{2}}=0.$$

 For $a\in \Z$, let $\C^H_{ar}$ be the one
dimensional module in $\cat_{\bar 0}$ where both $E$ and $F$ act by zero and $H$ acts by $ar$.
 For each $n \in \{0,\ldots,r-2\}$,  let $S_n$ be the usual
$(n+1)$-dimensional simple highest weight $\overline U_q^{H}\mathfrak{sl}(2)$-module with
highest weight $n$. The module $S_n$ has highest weight vector $s_0$ such that $Es_0=0 $ and $Hs_0=ns_0$.
Then
$\{s_0, s_1,\ldots, s_n\}$  is a basis of $S_n$ where $Fs_i=s_{i+1}$, $H.s_i=(n-2i)s_i$, $E.s_0=0=F^{n+1}.s_0$
and $E.s_i=\frac{\qn i\qn{n+1-i}}{\qn1^2}s_{i-1}$.
Every simple module of $\cat$ is isomorphic to exactly one of the
modules in the list:
\begin{itemize}
	\item  $S_n\otimes  \C^H_{ar}$, for $n=0,\cdots, r-2$ and $a\in \Z$,
	\item  $V_\alpha$  for  $\alpha\in(\C\setminus \Z)\cup
	r\Z$.
\end{itemize}

For $i\in \{0,...,r-2\}$, let $P_i$ be the projective and
indecomposable module with highest weight $2r-2-i$, defined in
Proposition 6.2 of \cite{CGP2}.    Moreover, any indecomposable projective weight module has a
highest weight, and such a module $P\in\cat_{\wb0}\cup\cat_{\wb1}$ with
highest weight $(k+2)r-i-2$ is isomorphic to $P_i\otimes \C^H_{kr}$.

\section{Categorical preliminaries.}

\subsection{Hermitian ribbon category}

Here we follow \cite[Section 5.1]{Tu}.  Let $\cat$ be a strict monoidal category.  A dagger, or conjugation,  on $\cat$, assigns to each morphism $f\maps V \to W$ a morphism $f^{\dagger} \maps W \to V$ such that \
\begin{equation}
  (f^{\dagger})^{\dagger}  =f, \quad (f\otimes g)^{\dagger}  = f^{\dagger} \otimes g^{\dagger} , \quad
 (f\circ g)^{\dagger}  = g^{\dagger}  \circ f^{\dagger}.
\end{equation}
These relations imply $ \Id_V^{\dagger} = \Id_V$.  In other words, $\dagger$ is an object preserving contravariant involution on $\cat$.

\begin{defi}\label{def:Hermition}
 A {\em Hermitian ribbon category} is a ribbon monoidal category $\cat$ equipped with a conjugation satisfying the following conditions:
\begin{enumerate}[(i)]
\item for any objects $V$, $W$ of $\cat$, we have
\begin{equation}
   c_{V\otimes W}^{\dagger} = (c_{V,W})^{-1},
\end{equation}

\item for any object $V$ of $\cat$, we have \footnote{We use an equivalent definition to the one Turaev uses for $\ev_V^{\dagger}$. }
\begin{equation}\label{eq:Hermitian+dual}
   \theta_V^{\dagger}= (\theta_V)^{-1},
 \quad
 \coev_V^{\dagger} = \ev_V c_{V,V^{\ast}}(\theta_V \otimes \Id_{V^{\ast}}),
\quad
 \ev_V^{\dagger} =
 (\Id_{V^{\ast}} \otimes \theta_V) c_{V,V^{\ast}} \coev_V .
\end{equation}
\end{enumerate}
\end{defi}
\subsection{Modified traces on the projective modules.}
Let $\Proj$ be the full subcategory of $\cat$ consisting of projective
$\Ubar$-modules.  The subcategory $\Proj$ is an ideal (see also
\cite{GKP1}).  That is, it is closed under retracts (i.e.\ if $W \in \Proj$ and
$\alpha: X \to W$ and $\beta: W \to X$ satisfy $\beta \circ \alpha =
\Id_{X}$, then $X \in \Proj$) and if $X$ is in $\cat$, and $Y$ is in $\Proj$,
then $X \otimes Y$ is in $\Proj$.

For any objects $V,W$ of $\cat$ and any endomorphism $f$ of $V\otimes
W$, set
\begin{equation}\label{E:trL}
\ptr_{L}(f)=(\ev_{V}\otimes \Id_{W})\circ(\Id_{V^{*}}\otimes
f)\circ(\tcoev_{V}\otimes \Id_{W}) \in \End_{\cat}(W),
\end{equation} and
\begin{equation}\label{E:trR}
\ptr_{R}(f)=(\Id_{V}\otimes \tev_{W}) \circ (f \otimes \Id_{W^{*}})
\circ(\Id_{V}\otimes \coev_{W}) \in \End_{\cat}(V).
\end{equation}

\begin{defi}\label{D:trace}  A \emph{trace on $\Proj$} is a family of linear functions
$$\{\mt_V:\End_\cat(V)\rightarrow K\}$$
where $V$ runs over all objects of $\Proj$, such that the following two
conditions hold.
\begin{enumerate}
\item  If $U\in \Proj$, and $W\in \ob$, then for any $f\in \End_\cat(U\otimes W)$, we have
\begin{equation}\label{E:VW}
\mt_{U\otimes W}\left(f \right)=\mt_U \left( \ptr_R(f)\right).
\end{equation}
\item  If $U,V\in \Proj$, then for any morphisms $f:V\rightarrow U $, and $g:U\rightarrow V$  in $\cat$, we have
\begin{equation}\label{E:fggf}
\mt_V(g\circ f)=\mt_U(f \circ g).
\end{equation}
\end{enumerate}
\end{defi}
There exists up to a scalar a unique trace on $\Proj$. It is non-degenerate (cf Theorem 5.5 of \cite{GKP3}), in the following way. Let
$V,W\in\cat$ with $V$ projective. Then the pairing
$\brk{\cdot,\cdot}_{V,W}: \Hom_\cat(W,V)\otimes\Hom_\cat(V,W)\to\C$
given by $$\brk{f,g}_{V,W}=\mt_V(fg)$$ is non-degenerate.  It is
symmetric in the following sense. If $W$ is also projective, then
\begin{equation}
  \label{eq:pairing}
  \brk{g,f}_{W,V}=\brk{f,g}_{V,W} \ .
\end{equation}
If $W$ is not projective, then we take Equation \eqref{eq:pairing} as a
definition.

The modified dimension $\md(M)$, of an object $M$, is the modified trace of the identity morphism of $M$.  The modified trace on $\cat$ in this paper is normalized by
$$\md(V_\alpha)=\frac{\md_0\qn\alpha}{\qn{r\alpha}}$$
for a fixed real number $\md_0$.
\section{Hermitian ribbon structure on quantum $\mathfrak{sl}_2$-modules}\label{s:hrs}

A map $f:V\to W$ between two complex vector spaces is called \emph{antilinear} if $f(av+bv')=\bar a f(v)+ \bar b f(v')$ where $\bar a$ and $\bar b$ are the complex conjugates of the complex numbers $a$ and $b$, respectively.
Consider the operation
$\dagger:  x\mapsto \dagger(x):=\dag x$
defined on generators of $\Ubar$ by
%
\[
 \dag{E}=F, \quad \dag{F}=E, \quad \dag{K} = K^{-1}, \quad \dag{H} = H^{} .
 \]
\begin{lemma}
  The operator $\dagger: \Ubar \to \Ubar$ induces an antilinear, antialgebra
  involution which is also a coalgebra antimorphism.  That is,
  $\text{for any } a \in \Q(q) \text{ and } x,y \in \Ubar,$
\[
  \dagp{ax}=\bar{a}\dag{x} \quad \dagp{xy}=\dag{y}\dag{x} \quad \dagp{\dag x} =x \quad \Delta(\dag{x}) =(\dagger\otimes\dagger)(\tau(\Delta x)).
\]
Furthermore, $S(\dag x)=\dagp{ S(x)}$, 
$(\dagger\otimes\dagger)\bp{R}=\tau(R^{-1})$, and $\dag{\theta}=\theta^{-1}$.
\end{lemma}
\begin{proof}
This is similar to Lemma 1.3 of \cite{Wen98}.
\end{proof}

A \emph{Hermitian form} on a $\C$-vector space $V$ is a function
 $f \colon V\times V\to \C$ such that
\begin{enumerate}
  \item $f(v,av'+bv'')=af(v,v')+bf(v,v'')$,
  \item $f(v,v')=\overline{f(v',v)}$,
\end{enumerate}
for all $v,v', v''\in V$ and $a,b\in\C$.  It follows that $f$ is
antilinear in the first coordinate.  The kernel of $f$ is
$\{v\in V: V^* \ni f(v,\cdot)=0 \}$, and we say $f$ is non-degenerate if its kernel
is $\{0\}$.

Let $f \colon V\otimes V\to \C$ be a non-degenerate Hermitian form on a finite
dimensional weight $\Ubar$-module $V$. Then we say $f$ is
\emph{compatible} with the antilinear antialgebra automorphism $\dagger$ if
$f(\dag{x}v,v')=f(v,xv')$ for all $x\in \Ubar$ and $v,v'\in V$.
Equivalently, this means $\rho_V(x)^\dagger=\rho_V(\dag{x})$ where $\rho_V(x)^\dagger$ denotes the Hermitian adjoint on Hermitian vector spaces.   In this case, we say
$V$ is a \emph{Hermitian} $\Ubar$-module with \emph{Hermitian
  structure $f$}.

\begin{lemma} \label{daglemma}
Let $f_V$ and $f_W$ be Hermitian structures on $\Ubar$-modules
$V$ and $W$, respectively.  Then there is a well-defined adjoint map $\dagger$, which is
an antilinear homomorphism
\[
\begin{array}{rcl}
  \dagger: \Hom_\cat(V,W)&\to&\Hom_\cat(W,V)\\g&\mapsto& g^\dagger ,
 \end{array}
\] uniquely defined by
$f_W(\cdot,g(\cdot))=f_V(g^\dagger(\cdot),\cdot)$.

Moreover, if $U$
is a third Hermitian module, and
$h\in \Hom_\cat(W,U)$, then $(hg)^\dagger=g^\dagger h^\dagger$.
\end{lemma}
\begin{proof}
First, since $f_V$ and $f_W$ are non-degenerate Hermitian forms, $g^\dagger$ is a well-defined linear map.  To see that this map is a $\Ubar$-module morphism, let $x\in  \Ubar$, $v\in V$ and $w\in W$. Then
$$
f_V(g^\dagger(xw),v)=f_W(xw,g(v))=f_W(w,g(\dag{x}v))=f_V(g^\dagger(w),\dag{x}v)=f_V(xg^\dagger(w),v).
$$
Since $v$ is any element of $V$ and $f_V$ is non-degenerate, it then follows that $g^\dagger(xw)=xg^\dagger(w)$.  Similar calculations show that the map $\dagger$ is antilinear and satisfies the last property of the lemma.
\end{proof}

If $V$ is a $\Ubar$ weight module, let $\con V=\{\bar v:v\in V\}$ be
the same real vector space with antilinear scalar multiplication, and
the action of $\Ubar$ on $\bar v\in\con V$ be given by
$x.\bar v=\wb{\dag{S(x)}.v}$.

A sesquilinear form $f$ on a
$\Ubar$-module $V$ is said to be compatible with $\Ubar$ if it is non-degenerate and for any
$x\in \Ubar$ and $v_1, v_2 \in V$, we have
$f(v_1,\rho_V(x)(v_2))=f(\rho_V(\dag{x})(v_1),v_2)$.
\begin{lemma}\label{L:sesq}
  Let $V$ be a $\Ubar$ weight module. Then there exists a compatible
  sesquilinear form $f$ on $V$ if and only if $\con V\simeq V^*$.
  Furthermore, if $V$ is simple, then $f$ can be chosen to be Hermitian and is
  unique up to a constant in $\R^*$.
\end{lemma}
\begin{proof}
  For the first part, we follow \cite{Kir96,Wen98}.  The sesquilinear
  form associated to an isomorphism $\vp: \con V\stackrel\sim\to V^*$ is given by
  $$f:V\times V\stackrel{\bar\Id\otimes\Id}\longrightarrow
  \con V\otimes V\stackrel{\vp\otimes\Id}\longrightarrow
  V^*\otimes V\stackrel\ev\longrightarrow \C .$$
  The compatibility follows from the following string of equalities:
  $$f(v_1,x.v_2)=\ev(\vp(\bar v_1)\otimes x.v_2)=\ev(S^{-1}(x).\vp(\bar v_1),v_2)
  =\ev(\vp(S^{-1}(x).\bar v_1),v_2)$$
  $$=\ev(\vp(\wb{\dag{x}.v_1}),v_2)=f(\dag{x}.v_1,v_2).$$

  Next, for a fixed basis $(e)$ of $V$, let $A$ denote the matrix of $f$:
  $A_{i,j}=f(e_i,e_j)$.  Let $\Mat_{(e)}(x)=M$ and denote by
  $M^*$ the conjugate transpose of $M$. Then
  $\Mat_{(e)}(\dag{x})= (A^{-1})^*M^*A^*$ and
  $M=\Mat_{(e)}(\dagp{\dag x})=(A^{-1})^*AMA^{-1}A^*$. Hence $A^{-1}A^*$
  commutes with the image of $\Ubar$ in $\End_\C(V)$.  Suppose now
  that $V$ is also simple.
  Then $A^{-1}A^*$ is a scalar because it is the matrix of a
  $\Ubar$-module morphism and $\End_\cat(V)=\C\Id_V$.  Let us write
  $A^{-1}A^*=\lambda^{2}$ then $(A^*)^*=A$ implies
  $\wb\lambda\lambda=1$.
 Then the matrix of $\lambda f$ is Hermitian as
  $(\lambda A)^*=\wb\lambda
  A^*=\wb\lambda\lambda^{2}A=|\lambda|^2(\lambda A)=\lambda A$.

  Finally, any two compatible sesquilinear forms $f,f'$ differ by an
  automorphism $g\in\Aut_\cat(V)$ such that $f'=f(\cdot,g(\cdot))$. So
  the last statement follows with $\Aut_\cat(V)=\C^*\Id$ and for $f$
  Hermitian, $\lambda f$ is Hermitian if and only if $\lambda\in\R$.
\end{proof}

If $A$ is a set of objects in $\cat$, then we define the \emph{category
  generated} by $A$ as the full subcategory of $\cat$ which has as
objects, all direct sums of retracts of all tensor products of the
form:
$$X_{1}\otimes X_{2}\otimes \cdots \otimes X_{p}\quad\text{ where }X_i\in A\cup A^*.$$

Each simple module $S_i$, for $i=0,\ldots, r-2$ has an indecomposable projective cover $P_i$.  The dimension of each $P_i$ is $2r$.  A detailed description of this module could be found in \cite[Proposition 6.1]{CGP2}.
A summary could be found in Figure \ref{fig:P_i}.  A vector $w_k^Y$, for $Y \in \{R,H,S,L\}$ has weight $k$ (under the action of $H$).

\begin{figure} [H]
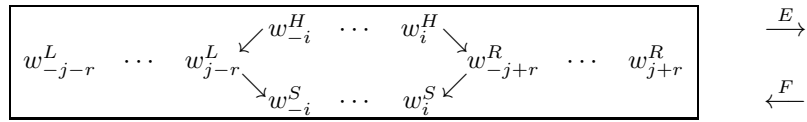

  \centering
 \[
\begin{array}{|ccccccccc|cc}
\cline{1-9}
  &&& \vech_{-i}& \cdots &\vech_i&&&&&\stackrel E{\longrightarrow}\\
\vecl_{-j-r}& \cdots &
\vecl_{j-r}
&&&& \vecr_{-j+r}& \cdots &\vecr_{j+r}&\ \ \ &\\
  &&& \vecs_{-i}& \cdots &\vecs_i&&&&&\stackrel F{\longleftarrow}\\
\cline{1-9}
\end{array}
\put(-220,8){$\swarrow$}\put(-219,-9){$\searrow$}
\put(-143,8){$\searrow$}\put(-143,-9){$\swarrow$}
\]

  \caption{The weight spaces structure of the module $P_i$ (here $j=r-2-i$).}
  \label{fig:P_i}
\end{figure}

Consider the subcategory $\catd$ of $\cat$ generated by the following set:
\begin{equation}\label{E:SetA}
\A=\left\{V_\alpha, S_n, P_i,  \C^H_{ar}  \; \Big| \; \alpha\in(\R\setminus \Z)\cup
	r\Z,  n,i \in\{0,\cdots, r-2\},  a\in \Z \right\}.
\end{equation}

\begin{prop}
  All projectives objects are in the Karoubi envelope (or idempotent completion) of the
  additive monoidal category generated by the simple objects.
\end{prop}

\begin{proof}
By \cite[Proposition 8.4]{CGP2}, one could generate all $P_k \otimes \C_{mr}^H$ where $k$ is even by decomposing
$(S_{r-1} \otimes \C_{mr}^H) \otimes S_{r-1}$.
Similarly, one could generate all $P_k \otimes \C_{mr}^H$ where $k$ is odd by decomposing
$(S_{r-2} \otimes \C_{mr}^H) \otimes S_{r-1}$.

All other 
indecomposable
projective objects are already simple.
\end{proof}

We will now show that each object of $\catd$ has a Hermitian structure and prove that $\catd$ is a Hermitian ribbon category.

\begin{prop} \label{prop:form-Valpha} Any simple module $V$ in $\A$ has
  a Hermitian structure.  Moreover, the form $(\cdot,\cdot)$ on $V$
  is uniquely determined by $(v_0,v_0) = 1$ where $v_0$ is a highest
  weight vector of $V$.
\end{prop}
\begin{proof}
  Let $V$ be a simple module in $\A$.  Since $\dag{H}=H$, we have $\con V$
  is a simple module with character that
  is the conjugate of that of $V^*$.  Since the weights of modules in
  $\A$ are real, $\con V$ and $V^*$ have the same character so they are
  isomorphic.  Then Lemma \ref{L:sesq} applies.
\end{proof}

 \begin{prop}\label{prop:form-Pi}
If  $i\in \{0,...,r-2\}$ then the projective indecomposable $P_i$ has a Hermitian form $(\cdot, \cdot )_{\alpha,\beta} $ making it a  Hermitian $\Ubar$-module where $\alpha$ and $\beta$ are two real non-zero parameters.
\end{prop}
 \begin{proof}
 Recall from \cite[Proposition 6.1]{CGP2} that the action of $\Ubar$ on the basis $\{w_k^Y \} $ of $P_i$ involves the quantities
$ \gamma_{n,k} = [k][n-k+1] $ .

Now define
\begin{equation}
(w_{-j-r}^L, w_{-j-r}^L)_{\alpha,\beta}=\alpha \hspace{.5in} (w_{-i}^H, w_{-i}^H)_{\alpha,\beta}=\beta
\ .
\end{equation}
If the form is Hermitian and $(xv,w)_{\alpha,\beta}=(v,x^{\dagger} w)_{\alpha,\beta}$, then
\begin{align*}
(w_{j-2k-r}^L, w_{j-2k-r}^L)_{\alpha,\beta} &= \left(\frac{E^{j-k}w_{-j-r}^L}{\displaystyle \prod_{m=k+1}^j -\gamma_{j,m}}, \frac{E^{j-k}w_{-j-r}^L}{\displaystyle \prod_{m=k+1}^j -\gamma_{j,m}} \right)_{\alpha,\beta} \\
&= \frac{1}{\displaystyle \prod_{m=k+1}^j \overline{\gamma_{j,m}} \gamma_{j,m}}(E^{j-k}w_{-j-r}^L , E^{j-k}w_{-j-r}^L)_{\alpha,\beta} \\
&= \frac{1}{\displaystyle \prod_{m=k+1}^j \overline{\gamma_{j,m}} \gamma_{j,m}}(w_{-j-r}^L , F^{j-k} E^{j-k}w_{-j-r}^L)_{\alpha,\beta} \\
&= \frac{1}{\displaystyle \prod_{m=k+1}^j -\overline{\gamma_{j,m}} }(w_{-j-r}^L , w_{-j-r}^L)_{\alpha,\beta}
\;\; = \;\; \frac{\alpha}{\displaystyle \prod_{m=k+1}^j -\overline{\gamma_{j,m}} } .
\end{align*}

Next note that
\begin{equation*}
(w_{-i}^S, w_{-i}^S)_{\alpha,\beta} = (Ew_{j-r}^L, Ew_{j-r}^L)_{\alpha,\beta}=(w_{j-r}^L,FE w_{j-r}^L)_{\alpha,\beta}=(w_{j-r}^L,Fw_{-i}^S)_{\alpha,\beta}=0 \ .
\end{equation*}
Similarly,

\begin{align*}
(w_{i-2k}^S, w_{i-2k}^S)_{\alpha,\beta} &= \left(\frac{E^{i-k}w_{-i}^S}{\displaystyle \prod_{m=k+1}^i \gamma_{j,m}}, \frac{E^{i-k}w_{-i}^S}{\displaystyle \prod_{m=k+1}^i \gamma_{j,m}} \right)_{\alpha,\beta} \\
&= \frac{1}{\displaystyle \prod_{m=k+1}^i \overline{\gamma_{i,m}} \gamma_{i,m}}(E^{i-k}w_{-i}^S , E^{i-k}w_{-i}^S)_{\alpha,\beta} \\
&= \frac{1}{\displaystyle \prod_{m=k+1}^i \overline{\gamma_{i,m}} \gamma_{i,m}}(w_{-i}^S , F^{i-k} E^{i-k}w_{-i}^S)_{\alpha,\beta} \\
&= \frac{1}{\displaystyle \prod_{m=k+1}^i \overline{\gamma_{i,m}} }(w_{-i}^S , w_{-i}^S)_{\alpha,\beta}
\;\; = \;\; 0 .
\end{align*}

Next note that
\begin{equation}
(Fw_{-i}^H,Fw_{-i}^H)_{\alpha,\beta}=(w_{j-r}^L,w_{j-r}^L)_{\alpha,\beta}
\end{equation}
and
\begin{equation}
(Fw_{-i}^H,Fw_{-i}^H)_{\alpha,\beta}
= (w_{-i}^H,EFw_{-i}^H)_{\alpha,\beta}
=(w_{-i}^H,Ew_{j-r}^L)_{\alpha,\beta}
=(w_{-i}^H,w_{-i}^S)_{\alpha,\beta} \ .
\end{equation}
Thus
\begin{equation}
(w_{-i}^H,w_{-i}^S)_{\alpha,\beta}=(w_{j-r}^L,w_{j-r}^L)_{\alpha,\beta}=
\frac{\alpha}{\displaystyle \prod_{m=1}^k -\overline{\gamma_{j,m}}} .
\end{equation}

Next note that
\begin{align*}
(w_{-j+r}^R, w_{-j+r}^R)_{\alpha,\beta} &= \left(\frac{E^{i+1}w_{-i}^H}{\displaystyle \prod_{m=1}^i \gamma_{i,m}}, \frac{E^{i+1}w_{-i}^H}{\displaystyle \prod_{m=1}^i \gamma_{i,m}} \right)  = \frac{1}{\displaystyle \prod_{m=1}^i \overline{\gamma_{i,m}}}(w_{-i}^H , F^{i+1}w_{-j+r}^R)_{\alpha,\beta} \\
&= \frac{1}{\displaystyle \prod_{m=1}^i \overline{\gamma_{i,m}}}(w_{-i}^H , w_{-i}^S)_{\alpha,\beta}  =\frac{\alpha}{\displaystyle \prod_{m=1}^i \overline{\gamma_{i,m}} \displaystyle \prod_{m=1}^j -\overline{\gamma_{j,m}}} .
\end{align*}
From this we compute
\begin{align*}
(w_{r-j+2k}^R,w_{r-j+2k}^R)_{\alpha,\beta} &= (E^k w_{r-j}^R, E^k w_{r-j}^R)_{\alpha,\beta}  =  (w_{r-j}^R, F^k w_{r-j+2k}^R)_{\alpha,\beta} \\
&= (w_{r-j}^R, \prod_{m=j}^k -\gamma_{j,m} w_{r-j}^R)_{\alpha,\beta} = \frac{\prod_{m=j}^k -\gamma_{j,m}}{\prod_{m=1}^i \overline{\gamma_{i,m}} \prod_{m=1}^j -\overline{\gamma_{j,m}}}\alpha .
\end{align*}

Next note that
\begin{align*}
(w_{i-2k}^H,w_{i-2k}^S)_{\alpha,\beta} &=
\left(\frac{E^{i-k}w_{-i}^H}{\displaystyle \prod_{m=1}^k \gamma_{i,m}}+\Gamma, \frac{E^{i-k}w_{-i}^S}{\displaystyle \prod_{m=1}^i \gamma_{i,m}} \right)_{\alpha,\beta}  = \frac{1}{\displaystyle \prod_{m=1}^k \overline{\gamma_{i,m}}}(w_{-i}^H , F^{i-k}w_{i-2k}^S)_{\alpha,\beta} \\
&= \frac{1}{\displaystyle \prod_{m=1}^k \overline{\gamma_{i,m}}}(w_{-i}^H , w_{-i}^S)_{\alpha,\beta}  = \frac{\alpha}{\displaystyle \prod_{m=1}^k \overline{\gamma_{i,m}} \prod_{m=1}^j -\overline{\gamma_{j,m}}}
\end{align*}
where in the first equality, $\Gamma$ contains $w^S$ terms which were already show to kill terms on the right.

Next we compute
\begin{equation}
(w_{i-2k}^H,E^{i-k}w_{-i}^H)_{\alpha,\beta}=(F^{i-k}w_{i-2k}^H,w_{-i}^H)_{\alpha,\beta}=(w_{-i}^H,w_{-i}^H)_{\alpha,\beta}=\beta \ .
\end{equation}
But
\begin{equation}
(w_{i-2k}^H,w_{i-2k}^H)_{\alpha,\beta} = (w_{i-2k}^H,e_{i-k}(\gamma_{i,i},\ldots,\gamma_{i,k+1})w_{i-2k}^H+
e_{i-k-1}(\gamma_{i,i},\ldots,\gamma_{i,k+1})w_{i-2k}^S)_{\alpha,\beta}
\end{equation}
where the $e_{i-k}$ and $e_{i-k-1}$ are elementary symmetric functions.  This could be verified with a straightforward induction argument.
Thus
\begin{equation}
(w_{i-2k}^H,E^{i-k}w_{-i}^H)_{\alpha,\beta}=e_{i-k}(w_{i-2k}^H,w_{i-2k}^H)_{\alpha,\beta}+e_{i-k-1}(w_{i-2k}^H,w_{i-2k}^S)_{\alpha,\beta} \ .
\end{equation}
Thus
\begin{equation}
(w_{i-2k}^H,w_{i-2k}^H)_{\alpha,\beta} =
\frac{\left( \beta - \frac{e_{i-k-1}(\gamma_{i,i},\ldots \gamma_{i,k+1} )\alpha}{\displaystyle \prod_{m=1}^k \overline{\gamma_{i,m}} \displaystyle \prod_{m=1}^j -\overline{\gamma_{j,m}}} \right)}{e_{i-k}(\gamma_{i,i},\ldots,\gamma_{i,k+1})} .
\end{equation}
 \end{proof}

\begin{rem}
Consider the submodule $\Delta_{j+r}$ of $P_i$
defined as follows
\begin{equation} \label{Deltadef}
\Delta_{j+r}=\{w_{j+r}^R, \ldots, w_{j-r}^R, w_i^S, \ldots, w_{-i}^S \} \ .
\end{equation}
It inherits the form $(\cdot,\cdot)_{\alpha,\beta}$ from $P_i$.  Note
that this form is degenerate on $\Delta_{j+r}$.  The submodule of
$\Delta_{j+r}$ spanned by the vectors $\{w_i^S, \ldots, w_{-i}^S \}$
is isomorphic to the simple module $S_i$ but is actually a totally
isotropic subspace of $P_i$.  So we should not expect an abelian
structure on a Hermitian category.
\end{rem}

\begin{prop}
The simple module $S_i$ is a quotient of $P_i$ by the radical of a specialization of the Hermitian form $(\cdot, \cdot)_{\alpha,\beta}$:
\[
S_i \cong P_i/rad(\cdot,\cdot)_{0,\beta} \ .
\]
\end{prop}

\begin{proof}
The preceding analysis shows that all vectors $w^L,w^S,w^R$ are in the radical, while the vectors $w^H$ are not.
The quotient of $P_i$ by this radical yields a module of dimension $i+1$ with highest weight $i$ which is isomorphic to $S_i$.
\end{proof}

Let $\catHe$ be the full subcategory of 
$\cat$ whose objects are Hermitian  $\Ubar$-modules.
Let $\catdHe$ be the full subcategory of $\catHe$
whose objects are in $\catd$.  We will need the
following lemmas to define a tensor product in $\catdHe$.
\begin{lemma}\label{L:indecHe}
  Any module in $\catd$ has a Hermitian structure.
\end{lemma}
\begin{proof}
  Using orthogonal direct sums, it is sufficient to show that any
  indecomposable module in $\catd$ has a Hermitian structure.  The
  category $\catd$ has indecomposable modules isomorphic to
  $V'=V\otimes\C^H_{kr}$ for some $(V,k)\in\A\times\Z$ (see
  \cite{CGP2} where the tensor decomposition rules are described).
  Proposition \ref{prop:form-Valpha} and Proposition
  \ref{prop:form-Pi} ensure that any $V\in\A$ has a Hermitian
  structure.  Choose such a Hermitian form.  As a complex vector
  space, $V \cong V\otimes\C^H_{kr}=V'$, but they are different as representations. For example,
  $\rho_{V'}(E)=q^{kr}\rho_{V}(E)$.  We can consider the Hermitian
  form on $V'$ given by $\bp{x,y}_{V'}=a\bp{x,q^{krH/2}y}_{V}$ where
  $a=\wb a^{-1}$ is a square root of $\brk{K^{-kr}}_V$.  Note that the expression
  $\brk{K^{-kr}}_V$ makes sense since the central element $K^{-kr}$
  acts as a scalar operator on $V$.  The form is Hermitian since
  $$\bp{y,x}_{V'}=a\bp{y,q^{krH/2}x}_{V}=\wb a\bp{a^{-2}q^{-krH/2}y,x}_{V}=\wb
  a\bp{q^{krH/2}y,x}_{V}=\wb{\bp{x,y}_{V'}}.$$ The second
  equality comes from $\dagp{q^{krH/2}}=q^{-krH/2}$ and
  $a=\wb a\, \wb a^{-2}$. The third equality follows from:
  \[
  \rho_V(a^{-2}q^{-krH/2})=\rho_V(K^{kr}q^{-krH/2})=\rho_V(q^{krH/2}).
  \]

Next we check the compatibility of this form with $\Ubar$.
  $$\bp{y,\rho_{V'}(E)x}_{V'}
  =a\bp{y,q^{kr}\rho_{V}(E)q^{krH/2}x}_{V}
  =a\bp{y,q^{krH/2}\rho_{V}(E)x}_{V}
  =a\bp{\rho_{V}(E^\dagger)y,q^{krH/2}x}_{V}$$
  $$
  =a\bp{\rho_{V}(F)y,q^{krH/2} x}_{V}
  =\bp{\rho_{V'}(F)y,x}_{V'}
  =\bp{\rho_{V'}(E^\dagger)y,x}_{V'} .
  $$
  Thus $\rho_{V'}(E)^\dagger= \rho_{V'}(E^\dagger)$. Similar
  computations for $F$, $K$ and $H$ show that the form is compatible with $\Ubar$.
\end{proof}
\begin{lemma}\label{L:catHeKS}
  Any Hermitian module $V$ in $\catdHe$ splits as an orthogonal direct
  sum of Hermitian indecomposable modules.
\end{lemma}
\begin{proof}
  We prove that $V$ has an indecomposable Hermitian submodule $P$,
  then the proof follows by induction since the non degeneracy of
  $(\cdot,\cdot)$ on $P$ implies that $P^\perp$ is a complementary submodule
  of $P$ in $V$.

  Consider a direct sum decomposition $V=\bigoplus_iW_i$ with $W_i$
  indecomposable and let $W'_i=(\bigoplus_{j\neq i}W_j)^\perp$.  Then the proof of
  Lemma \ref{L:sesq} implies that $\con {(W'_i)}\simeq W_i^*$.
  But as for any indecomposable module of
  $\catdHe$ we have $\con {(W_i)}\simeq W_i^*$ so $W'_i\simeq W_i$. If
  we choose such an isomorphism $f:W_i\to W'_i$, then
  $(f(\cdot),\cdot)$ is a compatible sesquilinear form on $W_i$.
  Lemma \ref{L:indecHe} ensures that $W_i$ has a compatible Hermitian
  form $(\cdot,\cdot)_{W_i}$. Thus there exists $g\in\Aut_\cat(W_i)$
  such that  $(f(\cdot),\cdot)=(g(\cdot),\cdot)_{W_i}$: Then
  $h=fg^{-1}:W_i\to W'_i$ is an isomorphism such that
  $(h(\cdot),\cdot)=(\cdot,\cdot)_{W_i}$ is Hermitian.  For
  $\alpha\in\R$, let $W_i''=(h+\e^{i\alpha}\Id)(W_i)$.  Fix a basis
  $(e)$ of $W_i$, $(e')=h((e))$ and $(e'')=(h+\e^{i\alpha}\Id)((e))$.
  Let $B=\Mat_{(e)}((\cdot,\cdot)_{W_i})$,
  $A=\Mat_{(e)}((\cdot,\cdot))$, $A'=\Mat_{(e')}((\cdot,\cdot))$ and
  $A''=\Mat_{(e'')}((\cdot,\cdot))$.  Then
  $A''_{ij}=(e''_i,e''_j)=(e'_i+\e^{i\alpha}e_i,e'_j+\e^{i\alpha}e_j)
  =A'_{ij}+A_{ij}+\e^{-i\alpha}\wb{B_{j,i}}+\e^{i\alpha}{B_{i,j}}
  =A'_{ij}+A_{ij}+2\cos\alpha B_{i,j}$ so that
  $A''=A+A'+2\cos\alpha B=B(B^{-1}(A+A')+2\cos\alpha)$ which is non-degenerate
  when $-2\cos\alpha$ is not an eigenvalue of
  $B^{-1}(A+A')$.  Then $W_i''$ is an indecomposable Hermitian
  submodule of $V$.
\end{proof}
Note that the above statement is false in $\catHe$. For
example, one can check that $V_{1+i}\oplus V_{1-i}$ has a Hermitian structure
but $V_{1\pm i}$ does not.


We will now introduce the notion of a half twist and give a construction of it in our category.  For related ideas, see \cite{ST}.

\newcommand{\dt}{\sqrt\theta}
\begin{defi}
  A half twist $\dt$ for $\cat$ is a natural
  isomorphism of the identity functor whose square is the twist.  That
  is for any $f:V\to W$, $\dt_W f=f\dt_V$,
  $\dt_V^2=\theta_V\in\End_\cat(V)$ and $\dt_{V^*}=\bp{\dt_V}^*$.  We
  also assume that $\dt_\unit=\Id_\unit$.
\end{defi}

\begin{prop}
  Any $\FK$-linear ribbon category over an algebraically closed field
  of characteristic $0$ has a half twist.
\end{prop}
\begin{proof}
  Consider a maximal block decomposition $\cat=\bigoplus_i\cat_i$
  where for any $(V,W)\in\cat_i\times \cat_j$,
  $\Hom_\cat(V,W)\neq0\implies i=j$.  Then the twist has a unique
  generalized eigenvalue $\theta_i\in\FK$ on $\cat_i$. That is for any
  $V\in\cat_i$, $\brk{\theta_V}_V=\theta_i$.  One can choose a square
  root $\dt_i$ of $\theta_i$ such that $\dt_{i^*}=\dt_i$ where
  $\cat_{i^*}$ is the block with the dual objects of $\cat_i$. Then a
  half twist is given by the formula:
  $$\forall V\in\cat_i, \dt_V= \dt_i\operatorname{sqr}(\theta_V\theta_i^{-1})$$
  where $\operatorname{sqr}$ is the square root function defined on an
  unipotent $x$ by
  $$\operatorname{sqr}(x)=1-\sum_{k\ge0}\frac{\bp{2k}!(1-x)^{k+1}}{2^{2k+1}k!\bp{k+1}!}.$$
  Note that $\theta_V\theta_i^{-1}$ is unipotent since
  $\theta_V\theta_i^{-1}-\Id_V=\theta_i^{-1}(\theta_V-\brk{\theta_V}\Id_V)$
  is nilpotent.   The naturality follows
  from the uniform choice of $\dt_i$ for all modules of $\cat_i$ and
  the self duality comes from $\theta_{V^*}=\bp{\theta_V}^*$.
\end{proof}
We now fix a half twist by its values on simples in \eqref{halftwistchart}.
Note that for any $k\in\Z$, $i=0,\ldots, r-2$ and $j=r-2-i$, $S_i\otimes\C_{kr}^H$ and $S_j\otimes\C_{(k+1)r}^H$ belong to the
same block since they are both
  simple composition factors in a Jordan-Holder series of
  $P_i\otimes\C_{kr}^H$.
\begin{equation} \label{halftwistchart}
\begin{array}{|c|c|c|c|c|}
  \hline
  V&V_\alpha&S_i&S_i\otimes\C_{kr}^H,\,k\in2\Z&S_i\otimes\C_{kr}^H,\,k\in2\Z+1\\
  \hline
  \brk{\theta_V}&q^{\frac{\alpha^2-(r-1)^2}2}&q^{\frac{i(i+2-2r)}2}
                &q^{\frac{i(i+2-2r)}2}&q^{\frac{j(j+2-2r)}2}\\
  \hline
  \brk{\dt_V}&q^{\frac{\alpha^2-(r-1)^2}4}&q^{\frac{i(i+2-2r)}4}
                &q^{\frac{i(i+2-2r)}4}&q^{\frac{j(j+2-2r)}4}\\
  \hline
\end{array}\,\mathsmall{
  \begin{array}{c}
    \alpha\in\Cp,\\i=0,\ldots, r-2\\j=r-2-i
  \end{array}
}
\end{equation}

Using this half twist, we can define an involutive isomorphism
(see also \cite{Wen98}) $\X: (\cat,\otimes) \to (\cat,\otimes^{op})$ in $\cat$ by
\begin{equation} \label{defofX}
\X_{V,W}=\bp{\dt_{W\otimes V}}^{-1}c_{V,W}(\dt_{V}\otimes \dt_{W}).
\end{equation}
Note that $\X$ is not a braiding, but we have the following lemma.
\begin{lemma}\label{L:X}
  For any $V,V',V''\in\cat$ and any morphisms $f,g$, we have the following equalities
  $$\X_{V',V}\X_{V,V'}=\Id,\qquad\X(f\otimes g)=(g\otimes f)\X, $$
  $$\X_{V'\otimes V,V''}(\X_{V,V'}\otimes\Id_{V''})
  =\X_{V,V''\otimes V'}(\Id_{V}\otimes\X_{V',V''}):
  V\otimes V'\otimes V''\to V''\otimes V'\otimes V,$$
  $$\X_{V,V^*} \coev_{V}=\tcoev_V\et  \ev_{V}\X_{V,V^*} =\tev_V \ .$$
  $$$$
\end{lemma}
\begin{proof}
In order to prove the first identity, we first recall the following fact
   \begin{equation} \label{crosstwistid}
   \theta_{V\otimes
    V'}=(\theta_{V}\otimes\theta_{V'})c_{V',V}c_{V,V'} \ .
    \end{equation}
We have
  \begin{align*}
    \X_{V',V}\X_{V,V'}
    &=\bp{\dt_{V\otimes V'}}^{-1}c_{V',V}(\dt_{V'}\otimes \dt_{V})
      \bp{\dt_{V'\otimes V}}^{-1}c_{V,V'}(\dt_{V}\otimes \dt_{V'})\\
    &=\bp{\dt_{V\otimes V'}}^{-1}\bp{\dt_{V\otimes V'}}^{-1}c_{V',V}(\dt_{V'}\otimes \dt_{V})
      c_{V,V'}(\dt_{V}\otimes \dt_{V'})\\
    &=\bp{\theta_{V\otimes V'}}^{-1}(\dt_{V}\otimes \dt_{V'})^2c_{V',V}
      c_{V,V'} \\
         &=\bp{\theta_{V\otimes V'}}^{-1}(\theta_{V}\otimes \theta_{V'}) c_{V',V}
      c_{V,V'} \\
      &= \bp{\theta_{V\otimes V'}}^{-1}  \bp{\theta_{V\otimes V'}} \\
      &= \Id_{V \otimes V'} \ .
  \end{align*}
  The second equality follows from the naturality of
  $\dt$. The third equality follows from the naturality of the braiding.
  The fifth equality follows from \eqref{crosstwistid}.

  The second identity of the lemma is also an easy consequence of
  these naturalities.

  For the third identity, we have that both sides are equal to
  \begin{align*}
    \bp{\dt_{V''\otimes V'\otimes V}}^{-1}(c_{V',V''}\otimes\Id)
    (\Id\otimes c_{V,V''})(c_{V,V'}\otimes\Id)
    (\dt_{V}\otimes \dt_{V'}\otimes \dt_{V''}) \ .
  \end{align*}
  For the fourth identity, we have the following equalities.
  \begin{align*}
    \X_{V,V^*} \coev_{V}
    &=\bp{\dt_{V^*\otimes V}}^{-1}c_{V,V^*}(\dt_{V}\otimes \dt_{V}^*)\coev_{V}\\
    &=c_{V,V^*}(\dt_{V}\otimes \dt_{V}^*) \coev_{V}\bp{\dt_{\unit}}^{-1}\\
    &=c_{V,V^*}(\dt_{V}^2\otimes \Id)\coev_{V}\bp{\dt_{\unit}}^{-1}\\
    &=c_{V,V^*}(\theta_{V}\otimes \Id)\coev_{V}\\
    &= \tcoev_{V} \ .
  \end{align*}
The second equality follows from the naturality of $\bp{\dt_{V^*\otimes V}}^{-1} $.
The third equality comes from moving $\dt_{V}^*$ in the second tensor factor to $\dt_{V} $ in the first factor (which is allowed in a ribbon category).
The final equality is also a standard identity in a ribbon category.
The last identity is proved in a similar way.
\end{proof}
Let $W_1,W_2\in\catHe$ and let $(\cdot,\cdot)_{W_i}$ be the Hermitian
structures on $W_i$.  Define $(\cdot,\cdot)_p$ on $W_1 \otimes W_2$ by
\begin{equation} \label{formpdef}
(x \otimes y , x' \otimes y' )_p = (x, x')_{W_1} (y , y' )_{W_2}
\end{equation}
for
$x,x'\in W_1$ and $y,y'\in W_2$.  The sesquilinear form $(\cdot,\cdot)_p$ is
Hermitian, non-degenerate but not compatible with $\dagger$.
We modify the pairing to obtain a Hermitian form compatible with $\dagger$ in the next result.

\begin{theo}\label{Th:tensor} Let $W_1,W_2\in\catdHe$.  Then $W_1 \otimes W_2$ is a
  Hermitian module with Hermitian structure given by
   $$(v, v') = (v, \tau(\X_{W_1 \otimes W_2} v'))_p.$$
   This gives $\catdHe$ the structure of a strict $\C$-linear tensor
   category.
\end{theo}
\begin{proof}
  Let $u\in \Ubar$, write $\Delta(u)=u_1\otimes u_2$ and consider
  $v=v_1\otimes v_2\in W_1 \otimes W_2$ and
  $v''=\X_{W_1 \otimes W_2} (v')=v_2''\otimes v_1''\in W_2 \otimes
  W_1$.  Note that we have omitted all summation symbols and will continue to do so throughout the course of the proof.
  Then we have
  $$(v, u.v')=(v, \tau(\X_{W_1 \otimes W_2} u.v'))_p
  =(v, \tau (u.v''))_p=(v_1,u_2v''_1)_{W_1}(v_2,u_1v''_2)_{W_2}$$$$
  =(\dag{u_2}v_1,v''_1)_{W_1}(\dag{u_1}v_2,v''_2)_{W_2} =(\dag{u}.v,
  \tau(v''))_p=(\dag{u}.v,
  v'),$$ where we used that
  $\Delta(\dag{u})=\dag{u_2}\otimes\dag{u_1}$.  Hence we have defined a
  compatible sesquilinear form on $W_1 \otimes W_2$.

  To prove that this form is Hermitian, first note that by Lemma
  \ref{L:catHeKS}, we can assume that $W_1$ and
  $W_2$ are indecomposable.  Then, consider a direct sum decomposition
  $W_1 \otimes W_2=\bigoplus_iV_i$ with
  $V_i$ indecomposable and the dual decomposition $V'_i=\qn{v'\in W_1
    \otimes W_2 | (v',\bigoplus_{j\neq
      i}V_j)=\qn0}$.

  Fix a factor $V_i$. Then
  $(\cdot,\cdot)$ is non-degenerate on $V'_i\times
  V_i$. Hence ${V'_i}^*\simeq \con{V_i}\simeq
  V_i^*$ so in particular, the half twist has the same value
  $\brk{\dt_{V_i}}=\brk{\dt_{V_i'}}$, which we denote by $\dt_0\in\C$. Let us call
  $\dt_1=\brk{\dt_{W_1}}$ and
  $\dt_2=\brk{\dt_{W_2}}$ and $\theta_i=\dt_i^2$. Then let $(v,v')\in
  V_i\times V'_i$ and write $R=\sum a\otimes b$.
 Then
  \begin{align*}
    (v',v)&=(v',\tau\X(v))_p \\
    	&= (v',\tau\dt_{W_2 \otimes W_1}^{-1}\tau R(\dt_{W_1}\otimes\dt_{W_2})(v))\\
          &=\dt_0^{-1}\dt_1\dt_2\bp{v',\sqr(\theta_0\Delta^{op}(\theta))
            R(\sqr(\theta_1^{-1}\theta^{-1})
            \otimes\sqr(\theta_2^{-1}\theta^{-1}))}_p\\
          &=\dt_0^{-1}\dt_1\dt_2\bp{\bp{\sqr(\wb\theta_1^{-1}\dagp{\theta^{-1}})
            \otimes\sqr(\wb\theta_2^{-1}\dagp{\theta^{-1}})}R_{21}^{-1}
            \sqr(\wb\theta_0\Delta(\dag{\theta}))
            )v',v}_p\\
          &=\bp{\dt_0\dt_1^{-1}\dt_2^{-1}\bp{\sqr(\theta_1\theta)
            \otimes\sqr(\theta_2\theta)}R_{21}^{-1}
            \sqr(\theta_0^{-1}\Delta(\theta^{-1}))
            )v',v}_p\\
          &=\bp{\tau(\dt_{W_2}^{-1}\otimes\dt_{W_1}^{-1})R^{-1}\tau
            \dt_{W_1 \otimes W_2}
            )v',v}_p=\bp{\tau\X^{-1}(v'),v}_p\\
          &=\bp{\tau\X(v'),v}_p=\wb{\bp{v,\tau\X(v')}_p}=\wb{\bp{v,v'}}.
  \end{align*}
  In the above equalities, the elements $\theta, \theta^{-1}$, etc., are all acting on modules.
  Note that the fifth equality above follows from the fact that $\theta_0, \theta_1, \theta_2$ all have norm one, and Lemma \ref{daglemma}.
\end{proof}
\begin{prop}
  The Hermitian structure from Theorem \ref{Th:tensor} defines an associative tensor product turning
  $\catdHe$ into a tensor category.  Furthermore, for any two morphisms
  $f,g\in\catdHe$, one has
  $$(f\otimes g)^\dagger=f^\dagger \otimes g^\dagger.$$
\end{prop}
\begin{proof}
  To show the associativity of the tensor product of Hermitian
  modules, consider the product Hermitian form on $V\otimes V'\otimes
  V''$ given by
  $$\bp{v_1\otimes v_1'\otimes v_1'',v_2\otimes v_2'\otimes v_2''}_t=
  \bp{v_1,v_2}\bp{v_1',v_2'}\bp{v_1'',v_2''}$$ Then the compatible
  Hermitian forms obtained by iterating Theorem \ref{Th:tensor} are related to
  this pairing by
  $$\bp{w_1,w_2}_{(V\otimes V')\otimes V''}=\bp{w_1,\tau_{13}\X_{V,V',V''}(w_2)}_t
  =\bp{w_1,w_2}_{V\otimes (V'\otimes
    V'')},$$ where $\X_{V,V',V''}=\X_{V'\otimes
    V,V''}(\X_{V,V'}\otimes\Id_{V''}) =\X_{V,V''\otimes
    V'}(\Id_{V}\otimes\X_{V',V''})$ and
  $\tau_{13}$ is the permutation $x\otimes y\otimes z\mapsto z\otimes
  y\otimes
  x$.  This gives the associativity of the tensor product of
  $\catdHe$.  The category also has a strict unit
  $\unit=\C$ with
  $(1,1)_\unit=1$ because $\X_{\unit,V}=\X_{V,\unit}=\Id_V$.

  The last statement is a consequence of the second property
  of Lemma \ref{L:X}.
\end{proof}

\begin{prop}\label{P:dual} If $V\in\catHe$, $V^*$ has a unique Hermitian structure
  determined by $\bp{\vp,\psi}=\sum_i\wb{\vp(e_i)}\psi(e'_i)$ where $(e_i)$
  and $(e'_i)$ are any dual bases of $V$ for the Hermitian form of $V$.
  Furthermore, If $V\in\catdHe$ then one has
  $$\tcoev_V^\dagger=\ev_V\et\tev_V^\dagger=\coev_V.$$
\end{prop}
\begin{proof}
  Fix a basis $e=(e_i)_i$ of $V$ and let $e^*=(e_i^*)_i$ be the dual
  basis of $V^*$.  Let $B_{ij}=\bp{e_i,e_j}$ and let $B'=B^{-1}$.
  Then $e'_i=\sum_j \wb{B'_{ij}}e_j$ since
  $\bp{\sum_j \wb{B'_{ij}}e_j,e_k}=\sum_j
  B'_{ij}\bp{e_j,e_k}=\delta_i^k$. Then the matrix
  $C_{ij}=\bp{e_i^*,e_j^*}$ is given by
  $$C_{ij}=\bp{e_i^*,e_j^*}
  =\sum_k\wb{e_i^*(e_k)}{e_j^*(e'_k)}=e_j^*(e'_i)=\wb{B'_{ij}}.$$ So
  $C=\wb B^{-1}$ is Hermitian.  In particular a change of basis would change $B$
  to $P^*BP$ and $C$ to
  $\bp{^tP^{-1}}^*C\,{^t}P^{-1}=\wb{P}^{-1}\wb
  B^{-1}\,\wb{P^*}^{-1}=\wb{P^*BP}^{-1}$ so it would define the same Hermitian
  form on $V^*$.

  Now for any vector $x$, one has
  $x=\sum_i(e'_i,x)e_i=\bp{(\cdot,x)\otimes\Id}(\sum_ie'_i\otimes e_i)$, so
  $$\bp{(\cdot,x)\otimes\Id}(\sum_ie'_i\otimes
  u.e_i)=u.x=\sum_i(e'_i,u.x)e_i=\sum_i(u^\dagger.e'_i,x)e_i=\bp{(\cdot,x)\otimes\Id}(\sum_iu^\dagger.e'_i\otimes
  e_i).$$ Then
  \begin{equation} \label{eqeeprime}
  \sum_ie'_i\otimes u.e_i=\sum_iu^\dagger.e'_i\otimes e_i
  \end{equation}
  because
  $\bp{\cdot,\cdot}$ is non-degenerate.  Finally,
  $$\bp{u.\vp,\psi}=\bp{\vp(S(u)\cdot),\psi}=\sum_i\wb\vp(S(u).e_i)\psi(e'_i)
  =\sum_i\wb\vp(e_i)\psi(S(u)^\dagger.e'_i)$$
  $$=\bp{\vp,\psi(S(u^\dagger)\cdot)}
  =\bp{\vp,u^\dagger.\psi}, $$
  where in the third equality we used \eqref{eqeeprime}.
  We now compute the Hermitian adjoint of
  $\tcoev_V$. Let $(v,\vp)\in V\times V^*$. Then
  $$\bp{\vp\otimes v,\tcoev_V}_{V^*\otimes V}=\bp{\vp\otimes v,\tau\X\tcoev_V}_p
  =\bp{\vp\otimes v,\tau\coev_V}_p=\sum_i\wb{\bp{e_i^*,\vp}}\bp{v,e_i}$$
  $$=\sum_i\wb{\vp(e_i')}\bp{v,e_i}=\wb{\vp(v)}=\bp{\ev_V(\vp\otimes v),1}_\unit.$$
  Hence $\tcoev_V^\dagger=\ev_V$.  Note that the second equality above follows from Lemma \ref{L:X}.

  Taking the Hermitian adjoint of the zigzag
  $(\tev\otimes\Id)(\Id\otimes\tcoev)=\Id_V$, we get that
  \[
  \Id_V=(\Id_V\otimes\tcoev_V^\dagger)(\tev_V^\dagger\otimes\Id_V)
  =(\Id_V\otimes\ev_V)(\tev_V^\dagger\otimes\Id_V) .
  \]
  Thus $\tev_V^\dagger=\coev_V$.
 \end{proof}

 \begin{prop} \label{daggeronc}
  Let $V,W\in\catHe$, then
  $${c_{V,W}}^\dagger=(c_{V,W})^{-1}\et{\theta_{V}}^\dagger=(\theta_V)^{-1}.
  $$
\end{prop}

\begin{proof}
Taking the Hermitian adjoint of the $R$-matrix in \eqref{eq:R}, and applying a permutation to the tensor factors, one recovers $R^{-1}$ from \eqref{eq:Rinverse}.
The first statement now follows immediately.
The second statement is contained in Lemma \ref{daglemma}.
\end{proof}

   We define the contravariant monoidal antilinear functor
   $\dagger{}: \catdHe\to \catdHe$ as the identity on each object and as the
   Hermitian adjoint on each morphism.  This functor is contravariant,
   monoidal, and antilinear.  We now have the following result.
 \begin{theo}$\catdHe$ is a \emph{Hermitian}
   ribbon category in the sense of Definition~\ref{def:Hermition}.
 \end{theo}
%

We conclude this section with some observations about how the Hermitian structure on $\catdHe$ interacts with the modified trace.
 \begin{lemma}\label{L:dagmt}
   Let $(W,(\cdot,\cdot))\in\catdHe$ with $W\in\Proj$. Then for any $f\in\End_\cat(W)$,
   $$\mt_W(f^\dagger)=\wb{\mt_W(f)}.$$
 \end{lemma}
 \begin{proof}
     By
   Proposition \ref{P:dual} the dagger functor commutes with the
   partial trace of morphisms.  It follows that the family of linear
   maps $\bp{\mt^\dagger_V}_{V\in\Proj}$ given by:
   $$
   \begin{array}{rcl}
     \mt^\dagger_V:\End_\cat(V)&\to&\C\\
     f&\mapsto&\mt^\dagger_V(f):=\wb{\mt_V(f^{\dagger})}
   \end{array}
   $$
   is a modified trace on $\Proj$. By unicity of the modified trace,
   $\mt^\dagger=\lambda \mt$ for some $\lambda\in \C$.  Now for any
   $\alpha\in(\C\setminus \Z)\cup r\Z$,
   $\mt(\Id_{V_\alpha})=\mt^\dagger(\Id_{V_\alpha})=\md(V_\alpha)\in\R$
   since we choose $\md_0\in\R$.
 \end{proof}
\begin{lemma}
  For any objects $V,W$ of $\catdHe$ with $V$ or $W$ projective, the pairing
  $$\brk{f,g} \mapsto \mt_V( f^\dagger g)= \mt_W(g f^\dagger  ):
  \Hom_\catd(V,W)\times \Hom_\catd(V,W)\to \C$$
 (when both make sense) is a non-degenerate Hermitian pairing.
\end{lemma}
\begin{proof}
  Since the mapping $\Hom_\catd(V,W)\to \Hom_\catd(W,V)$ given by
  $f\mapsto f^\dagger$ is an involution, the pairing of the lemma
  is non-degenerate.  The Hermitian symmetry
  follows 
  from
  $$\brk{g,f} =\mt_V(g^\dagger f )=\mt_V((f^\dagger g )^\dagger )=\wb{\mt_V( f^\dagger g)}=\overline{\brk{f,g}}$$
  where the second to last equality follows from Lemma \ref{L:dagmt}.
\end{proof}

\begin{example}
Let $\alpha\in \R$ then
$$\brk{\ev_{V_\alpha}, \ev_{V_\alpha}}=\mt_{V_\alpha}(\Id_{V_\alpha})=\md_0\frac{\sin(\alpha\pi/r)}{\sin( \alpha\pi)}$$ which is not positive for all simples labelled by $(\alpha+1-r,\alpha+3-r,...,\alpha+r-1)$.  These simples all lie in the same graded piece of the category.
\end{example}

\section{Implications for non-semisimple TQFT}
In this section we assume $r\notin4\Z$ (see \cite{BCGP1}) so that
$\cat$ is a relative $\C^*$-modular category.  We will show that the TQFT constructed in \cite{BCGP1} has a Hermitian structure.

%
%
%

%
%
\subsection{Hermitian structure on decorated cobordisms}
We follow Turaev \cite[I.II.5.1]{Tu} and define the dagger of a $\catdHe$-colored
ribbon graph $T$ as the following transformation: invert the
orientation in the surface of $T$, reverse directions of bands
and annuli of $T$, exchange bottom and top bases of coupons and
replace the color $f$ of a coupons with $f^\dagger$.  The colors of
edges do not change.

The boundary of a $\catdHe$-colored ribbon graph is a set of
$\catdHe$-colored framed points where a framed point $p=(V,\ve)$ is a
point equipped with a sign $\ve$, a non-zero vector (its framing)
tangent to the surface $T$ but not to the directions of the band, and
a color which is an object $V\in \catdHe$. Let $F(p)=V$ if
$\ve=+$ and $F(p)=V^*$ if $\ve=-$.  The conjugate $\wb p$ of a
$\catdHe$-colored framed point $p$ is obtained by changing the sign
and framing to their opposites.  One easily check that
$\partial (T^\dagger)=\wb{\partial T}$.

Recall the category of decorated cobordism introduced in \cite{BCGP1}.
Here we replace the ribbon category with $\catdHe$ where we restrict to
$\R/2\Z$ valued cohomology classes (instead of $\C/2\Z$) and we only
admit as objects {\em admissible} decorated surfaces defined as follows.
\begin{defi}[Objects of $\Cob$]
  A \emph{decorated surface} is a $4$-tuple $\dSu =(\Su,\{p_i\},\coh,
  {\La})$ where:
\begin{itemize}
\item $\Su$ is a closed, oriented surface which is
  an ordered disjoint union of connected surfaces each having a distinguished base point $*$;
\item $\{p_i\}$ is a finite (possibly empty) set of $\catdHe$-colored
  framed points with framing tangent to the surface $\Su$;
\item $\coh\in H^1(\Su\setminus \{p_1,\ldots, p_k\}, *;\R/2\Z)$ is a
  cohomology class;
\item {\em compatibility condition}: letting $\coh( m_i)=a_i\in\R/2\Z$
  where $m_i$ is a positively oriented circle around $p_i$, then we
  require $F(p_i)\in\catdHe_{a_i}$;
\item ${\La}$ is a Lagrangian subspace of $H_1(\Su;\R)$.
\end{itemize}
In what follows we will restrict our attention to admissible surfaces
that have the additional property that each component of $\dSu$ has
either a $\catdHe$-colored framed point $p$ with $F(p)\in\Proj$ or it
contains a closed curve $\gamma$ such that $\coh(\gamma)\notin\Z/2\Z$.
\end{defi}

\begin{rem}
Note that this restriction is not just a cosmetic
  simplification. The non admissible surfaces considered in
  \cite{BCGP1} would not lead to a Hilbert space valued TQFT.
\end{rem}

\begin{defi}[Morphisms of $\Cob$]\label{D:Cobordisms}
  Let ${\dSu}_\pm=(\Su_\pm,\{p_i^\pm\},\coh_\pm,{\La}_\pm)$ be admissible,
  decorated surfaces.  A \emph{decorated cobordism} from ${\dSu}_-$ to
  ${\dSu}_+$ is a $5$-tuple $\dM=(M,T,f, \coh,n)$ where:
\begin{itemize}
\item $M$ is an oriented $3$-manifold with boundary $\partial M$;
\item $f: \overline{\Su_-}\sqcup \Su_+\to \partial M$ is a
  diffeomorphism preserving the orientation, and the image under
  $f$ of the base points of $\overline{\Su_-}\sqcup \Su_+$ is denoted by $*$;
\item $T$ is a $\catdHe$-colored ribbon graph in $M$ such that
  $\partial T=\{\wb{f(p_i^-)}\}\cup \{f(p_i^+)\}$;
\item $\coh\in H^1(M\setminus T,*;\R/2\Z)$ is a cohomology class
  relative to the base points on $\partial M$, such that the
  restriction of $\coh$ to
  $(\partial M\setminus \partial T)\cap \Su_\pm$ is
  $(f^{-1})^*(\coh_{\pm})$;
\item the coloring of $T$ is compatible with $\coh$, i.e. each
  oriented edge $e$ of $T$ is colored by an object in
  $\catdHe_{\coh(m_e)}$ where $m_e$ is the oriented meridian of $e$;
\item $n$ is an arbitrary integer
  called the {\em signature-defect} of $\dM$.
\end{itemize}
We can summarize the first four items by saying that
$\partial\dM=\dSu_-^*\sqcup \dSu_+$ where the dual of a decorated
surface $\dSu =(\Su,\{p_i\},\coh, {\La})$ is defined to be
$\dSu^*=(\wb\Su ,\wb{\{p_i\}},\coh,{\La})$.

For cobordisms, we only consider admissible, decorated cobordisms that
have the additional property that each component of $M$ contains
either a component of $T$ with a $\Proj$-colored edge or it contains a
closed curve $\gamma$ such that $\coh(\gamma)\notin\Z/2\Z$.  This
condition is automatically satisfied by components of $\dM$ with
non-empty boundary. This restriction is also in \cite{BCGP1}.  Hence
morphisms of $\Cob$ are orientation preserving diffeomorphism classes of
admissible decorated cobordisms.
\end{defi}

\begin{prop}
  Let $\dM=(M,T,f, \coh,n):\dSu_-\to\dSu_+$ be a cobordism in $\Cob$.
  Then the following defines a cobordism $\dM^\dagger:\dSu_+\to\dSu_-$
  in $\Cob$:
  $$\dM^\dagger=(\wb M,T^\dagger,\wb f, \coh,-n),$$
  where $\wb M$ is $M$ with opposite orientation, and
  $\wb f=\overline{\Su_+}\sqcup \Su_-\to \partial \wb M$ is the same
  set-theoretic map as $f$.
  Furthermore, with the above assignment, $\Cob$ is a Hermitian ribbon
  category.
\end{prop}
\begin{proof}
  The first statement follows from the fact that
  $\partial\dM=\dSu_-^*\sqcup\dSu_+$ implies
  $\partial\bp{\dM^\dagger}=\dSu_+^*\sqcup\dSu_-$.  The category
  $\Cob$ has the ordered disjoint union as tensor product and it is
  obvious that
  $(\dM_1\circ\dM_2)^\dagger=\dM_2^\dagger\circ\dM_1^\dagger$, and
  $(\dM_1\sqcup\dM_2)^\dagger=\dM_1^\dagger\sqcup\dM_2^\dagger$.  The
  braiding in $\Cob$ is symmetric. The pivotal structure is given
  by the cylinder $\dSu\times[0,1]$ that gives four morphisms
  $\emptyset\to\dSu\otimes\dSu^*$, $\emptyset\to\dSu^*\otimes\dSu$,
  $\dSu^*\otimes\dSu\to\emptyset$ or $\dSu\otimes\dSu^*\to\emptyset$
  which satisfy the zig-zag relations.  Finally,
  \eqref{eq:Hermitian+dual} follows because the twist is trivial in
  $\Cob$ and the cylinder $\dSu\times[0,1]$ is positively
  diffeomorphic to itself with opposite orientation by the
  diffeomorphism $\Id_\Su\times(t\mapsto1-t)$.
\end{proof}
Note that the last diffeomorphism of cylinders in the proof
generalizes to the non-compact case of a $\catdHe$-colored ribbon
graph $T$ in $\R^2\times[0,1]$.  In this case, the dagger of this
3-manifold is identified via the reflection through the plane
$\R^2\times\qn{\frac12}$ with the graph $T^\dagger$ embedded
upside-down in $\R^2\times[0,1]$ with standard orientation.  This
transformation is compatible with the Reshetikhin-Turaev functor $F$
in the sense that $F(T^\dagger)=F(T)^\dagger$.
\subsection{The $3$-manifold invariant}
If $\dM=(M,T,f, \coh,n)$ is a closed decorated manifold (i.e
$\dM\in\End_\Cob(\emptyset)$) which is connected, a surgery
presentation of $\dM$ is a $\cat$-colored ribbon graph
$T\cup L\subset S^3$ such that $M$ is obtained from $S^3$ by surgery
on $L$ and each component of $L$ is colored by a Kirby color
$\Omega_\alpha=\sum_{k=0}^{r-1}\qd(\alpha+1-r+2k)V_{\alpha+1-r+2k}$ of
degree equal to the value of the cohomology class on its meridian.
Then the invariant of $\dM$ is given by
\begin{equation}
  \label{eq:Zrconnected}
  \Zr(\dM)=\eta\lambda^{m+n}\dep^{-\sigma}F'(L\cup T)
\end{equation}
where $m\in\N$ is the number of components of $L$, $\sigma\in\Z$ is the
signature of the linking matrix $\lk(L)$ and
\begin{equation}
  \label{eq:delta}
  \dep=q^{-\frac32}\e^{-i(s+1)\pi/4}
 \text{ where } s \text{ is in } \{1,2,3\} \text{ with } s \equiv r \text{ mod }4,
\end{equation}


\begin{equation}
\lambda= \frac{\sqrt{r'}}{r^2}  , \quad \quad
\eta=\frac{1}{r \sqrt{r'}} .
\end{equation}
The invariant is extended multiplicatively for disjoint unions.


\begin{lemma}  \label{lem:dagofZ}
Let $\dM$ be a closed decorated manifold then
\[
\mathsf{Z} (\dM^\dagger) =
\overline{\mathsf{Z}(\dM)}
\]
\end{lemma}

\begin{proof}
  If a framed link $L$ gives rise via surgery to the manifold $M$,
  then the manifold $\overline{M}$ may be constructed from
  $\overline{L}$ where $\overline{L}$ is the mirror image of $L$.  It
  follows from the conjugation on the underlying Hermitian category
  that $F'(\overline{L}\cup T^\dagger)=\overline{F'(L\cup T)}$.

  If the linking matrix of $L$ has signature $\sigma$, then the
  signature of the linking matrix of $\overline{L}$ is $-\sigma$.
  Since $\lambda$ and $\eta$ are real and the modulus of $\delta$ is
  $1$, the lemma follows.
\end{proof}

\begin{cor} \label{cor:Hersymcob}
  The pairing $\Cob(\emptyset,\dSu)\times\Cob(\emptyset,\dSu)\to\C$, given by
  $$(\dM_1,\dM_2)\mapsto \mathsf{Z} (\dM_1^\dagger\circ\dM_2)\in\C$$
  has Hermitian symmetry.
\end{cor}

\begin{proof}
This follows directly from Lemma \ref{lem:dagofZ}.
\end{proof}

\subsection{The $(2+1)$-TQFT}
Recall $V_0$ is the simple projective module with highest weight $r-1$ and $\C^H_{kr}$ (for $k \in \Z$) are the one-dimensional modules described in Section \ref{S:QUantSL2H}. Let $\hS_k$ be the decorated sphere defined in \cite{BCGP1} colored with points $U$, where
\begin{equation}
U =
\begin{cases}
((V_0,1),(\C^H_{kr},1),(V_0,-1)) & \text{ if } k \neq 0, \\
((V_0,1),(V_0,-1)) & \text{ if } k=0. \\
\end{cases}
\end{equation}
Here all modules are enhanced with their preferred Hermitian structure
(see Section \ref{s:hrs}).

Now we define the state space associated to a decorated surface in the following way:
$$\V(\dSu)=\Span_\C\qn{\Hom_{\Cob}(\emptyset,\dSu)}/K_{\dSu}$$
where $K_{\dSu}$ is the right kernel of the bilinear pairing given on
generators
$$
\begin{array}{ccl}
  \Span_\C\qn{{\Cob}(\dSu,\emptyset)}\otimes
  \Span_\C\qn{{\Cob}(\emptyset,\dSu)}&\to&\C\\ {}
  [{\dM_1}]\otimes[\dM_2]&\mapsto& \mathsf{Z} (\dM_1\circ\dM_2)
\end{array}
$$
$$\VV(\dSu)=\bigoplus_{k\in\Z}\VV_k(\dSu)\text{ where }\VV_k(\dSu)=\V(\dSu\sqcup\hS_k) .$$

These state spaces are part of a $(2+1)$-TQFT constructed in \cite{BCGP1}.  Since $\dagger:{\Cob}(\emptyset,\dSu)\to {\Cob}(\dSu,\emptyset)$ is bijective, the pairing described in Corollary \ref{cor:Hersymcob} descends to a non-degenerate Hermitian pairing on the state spaces $ \VV(\dSu)$ of the TQFT.
For details, see \cite[Proposition 4.28, Definition 5.3]{BCGP1}.

\begin{theo}
The TQFT $(\VV,\mathsf{Z}) $ is Hermitian.
More specifically, for any decorated surface $\dSu$, there is a non-degenerate Hermitian pairing
\[
\langle \cdot, \cdot \rangle_{\VV(\dSu)}
\]
and
for any $y \in \VV(\dSu_-)$ and for any $x \in \VV(\dSu_+)$, and
for any decorated cobordism $\dM : \dSu_- \to \dSu_+$ between decorated surfaces, there is an equality
\begin{equation} \label{adjofsurface}
\brk{ x,\VV(\dM)(y) }_{\VV(\dSu_+)}
=
\brk{  \VV(\dM^\dagger)(x),y }_{\VV(\dSu_-)}
.
\end{equation}
\end{theo}

\begin{proof}
The fact that the pairing is non-degenerate and Hermitian follows from the discussion above.

The equality \eqref{adjofsurface} follows from the functoriality of the TQFT and Lemma \ref{lem:dagofZ}.
The details follow as in \cite[Theorem III.5.3]{Tu}.
\end{proof}

The mapping class group action in $\Cob$ is given through mapping
cylinders: if $f:\dSu_1\to\dSu_2$ is a diffeomorphism, the mapping
cylinder of $f$ is the decorated cobordism from $\Su_1$ to $\Su_2$
given by
$M_f=(\Su_2\times[0,1],\{p_i^2\}\times[0,1],f\times\{0\}\sqcup
\Id\times\{1\},\pi^*(\coh_2),0)$.  The mapping cylinder construction
is functorial: $M_f\circ M_g=M_{fg}$.
\begin{prop} \label{prop:mapping}
  If $f$, $M_f$ are as above, $M_f^\dagger=M_{f^{-1}}$. In particular
  the TQFT $\VV$ induces projective representations of the mapping
  class group in the group of indefinite unitary matrices.
\end{prop}
\begin{proof}
 The map
 $f\times (t\mapsto1-t)
 :M_{f^{-1}}\to M_f^\dagger$ is an isomorphism, and thus the
  two cobordisms are equal in $\Cob$.
\end{proof}

%
\providecommand{\bysame}{\leavevmode\hbox to3em{\hrulefill}\thinspace}
\providecommand{\MR}{\relax\ifhmode\unskip\space\fi MR }
\providecommand{\MRhref}[2]{%
  \href{http://www.ams.org/mathscinet-getitem?mr=#1}{#2}
}
\providecommand{\href}[2]{#2}

\end{document}